\documentclass[a4paper, 10pt, parskip=half]{scrartcl}

\usepackage[utf8]{inputenc}
\usepackage[T1]{fontenc}
\usepackage{lmodern}
\usepackage{csquotes}
\usepackage{amsmath}
\usepackage{amssymb}
\usepackage{amsthm}
\usepackage{amsfonts}
\usepackage{dsfont}
\usepackage{mathtools}
\usepackage{marginnote}
\usepackage[dvipsnames]{xcolor}
\usepackage{hyperref}
\hypersetup{%
  colorlinks=true,
  linkcolor=black,
  citecolor=black,
  urlcolor=blue,
  pdftitle={Corners and collapse: Some simple observations concerning critical masses and boundary blow-up in the fully parabolic Keller--Segel system},
  pdfauthor={},
  pdfkeywords={},
  bookmarksopen=true,
}

\usepackage[numbers, sort&compress]{natbib}
\newcommand{\R}{\mathbb{R}}

\newcommand{\N}{\mathbb{N}}

\newcommand{\mc}[1]{\mathcal{#1}}

\newcommand{\ur}[1]{\mathrm{#1}}
\newcommand{\ure}{\ur{e}}

\ifdefined\labelenumi%
  \renewcommand{\labelenumi}{(\roman{enumi})}
  
\fi

\DeclareMathOperator{\supp}{supp}

\newcommand{\defs}{\coloneqq}

\newcommand{\nea}{\nearrow}

\newcommand{\ol}{\overline}

\newcommand{\wt}{\widetilde}

\newcommand{\ds}{\,\mathrm{d}s}

\newcommand{\ddt}{\frac{\mathrm{d}}{\mathrm{d}t}}

\newcommand{\hp}{\hphantom}
\newcommand{\pe}{\mathrel{\hp{=}}}

\newcommand{\tmax}{T_{\max}}

\newcommand{\intom}{\int_\Omega}

\newcommand{\intntom}{\int_0^T \int_\Omega}
\newcommand{\intnstom}{\int_0^t \int_\Omega}

\newcommand{\Ombar}{\ol \Omega}

\newcommand{\loc}{\mathrm{loc}}

\newcommand{\leb}[2][\Omega]{\ensuremath{L^{#2}(#1)}}

\newcommand{\sob}[3][\Omega]{\ensuremath{W^{#2, #3}(#1)}}

\newcommand{\con}[2][\Ombar]{\ensuremath{C^{#2}(#1)}}

\newcommand{\dual}[1]{\ensuremath{(#1)^\star}}

\newcommand{\Lom}[1]{L^{#1}(\Omega)}
\newcommand{\norm}[2][]{\|#2\|_{#1}}
\newcommand{\set}[1]{\{#1\}}

\usepackage{newunicodechar}
\newunicodechar{ε}{\varepsilon}
\newunicodechar{ℝ}{\R}
\newunicodechar{π}{\pi}

\makeatletter
\renewenvironment{proof}[1][\proofname]{\par
  \pushQED{\qed}%
  \normalfont \topsep0\p@\relax
  \trivlist
  \item[\hskip\labelsep\scshape
  #1\@addpunct{.}]\ignorespaces
}{%
  \popQED\endtrivlist\@endpefalse
}
\makeatother

\newtheorem{base}{Base}[section]
\numberwithin{equation}{section}

\newtheorem{theorem}[base]{Theorem} \newtheorem*{theorem*}{Theorem}
\newtheorem{lemma}[base]{Lemma} \newtheorem*{lemma*}{Lemma}
\newtheorem{prop}[base]{Proposition} \newtheorem*{prop*}{Proposition}
 \newtheorem*{cor*}{Corollary}

\theoremstyle{definition}
\newtheorem{remark}[base]{Remark} \newtheorem*{remark*}{Remark}
\newtheorem{definition}[base]{Definition} \newtheorem*{definition*}{Definition}
 \newtheorem*{example*}{Example}
 \newtheorem*{cond*}{Condition}

\setkomafont{title}{\normalfont\Large}
\title{Corners and collapse: Some simple observations concerning critical masses and boundary blow-up in the fully parabolic Keller--Segel system}

\usepackage{authblk}

\author[1]{Mario~Fuest\footnote{e-mail: fuest@ifam.uni-hannover.de}}
\author[1]{Johannes~Lankeit\footnote{e-mail: lankeit@ifam.uni-hannover.de}}
\affil[1]{Leibniz Universität Hannover, Institut für Angewandte Mathematik, Welfengarten 1, 30167 Hannover, Germany}
\date{}

\begin{document}
\maketitle

\KOMAoptions{abstract=true}
\begin{abstract}
\noindent
  Our main result shows that the mass $2\pi$ is critical for the minimal Keller--Segel system
  \begin{align}\label{prob:abstract}\tag{$\star$}
    \begin{cases}
      u_t = \Delta u - \nabla \cdot (u \nabla v), \\
      v_t = \Delta v - v + u,
    \end{cases}
  \end{align}
  considered in a quarter disc $\Omega = \{\,(x_1, x_2) \in \mathbb R : x_1 > 0, x_2 > 0, x_1^2 + x_2^2 < R^2\,\}$, $R > 0$, in the following sense:
  For all reasonably smooth nonnegative initial data $u_0, v_0$ with $\int_\Omega u_0 < 2\pi$,
  there exists a global classical solution to the Neumann initial boundary value problem associated to \eqref{prob:abstract},
  while for all $m > 2 \pi$ there exist nonnegative initial data $u_0, v_0$ with $\int_\Omega u_0 = m$ so that the corresponding classical solution of this problem blows up in finite time. \\[0.5pt]
  At the same time, this gives an example of boundary blow-up in \eqref{prob:abstract}.\\[0.5pt]
  Up to now, precise values of critical masses had been observed in spaces of radially symmetric functions or for parabolic--elliptic simplifications of \eqref{prob:abstract} only. \\[0.5pt]
 \textbf{AMS Classification (2020):} 35B30 (primary); 35A01, 35B44, 35K55, 92C17 (secondary)
\end{abstract}

\section{Introduction}
We study critical mass phenomena of solutions to the minimal fully parabolic Keller--Segel model
\begin{align}\label{prob}
  \begin{cases}
    u_t = \Delta u - \nabla \cdot (u \nabla v) & \text{in $\Omega \times (0, T)$}, \\
    v_t = \Delta v - v + u                     & \text{in $\Omega \times (0, T)$}, \\
    \partial_\nu u = \partial_\nu v = 0        & \text{on $\partial \Omega \times (0, T)$}, \\
    u(\cdot, 0) = u_0, v(\cdot, 0) = v_0       & \text{in $\Omega$},
  \end{cases}
\end{align}
in bounded domains $\Omega$.
This system has been suggested in \cite{KellerSegelInitiationSlimeMold1970} to model aggregation behavior of certain slime mold amoebae and has,
along with many variants, since then been extensively studied,
see for instance the surveys \cite{Horstmann1970PresentKeller2003}, \cite{BellomoEtAlMathematicalTheoryKeller2015} and \cite{LankeitWinklerFacingLowRegularity2019}.

Before discussing qualitative properties of solutions of \eqref{prob},
we note that \eqref{prob} is locally well-posed also in non-smooth domains.
While for a four-component system introduced in \cite{KellerSegelInitiationSlimeMold1970}, a well-posedness result has recently been obtained in \cite{HorstmannEtAlFullKellerSegel2018},
for \eqref{prob} it mainly suffices to reference the classical work \cite{GajewskiZachariasGlobalBehaviourReactiondiffusion1998}.
That is done in Section~\ref{sec:local_ex}, where we prove
\begin{prop}\label{prop:local_ex}
  Suppose that
  \begin{align}\label{eq:intro:omega}
    &\Omega \text{ is a piecewise $C^{1+\alpha}$ bounded domain in $\R^2$ for some $\alpha \in (0, 1)$} \notag \\
    &\hphantom{\Omega} \text{ with a finite number of vertices and with nonvanishing interior angles}
  \end{align}
  (meaning that $\Omega$ is a bounded planar domain of class $\Sigma^{1,\alpha}$ for some $\alpha \in (0, 1)$ in the sense of \cite[Definition~2.1]{CianchiMoserTrudingerInequalitiesBoundary2005})
  and
  \begin{align}\label{eq:intro:init_reg}
    0 \le u_0 \in \con0
    \quad \text{as well as} \quad
    0 \le v_0 \in \bigcup_{p > 2} \sob1p.
  \end{align}
  Then there exist $\tmax = \tmax(u_0, v_0) \in (0, \infty]$ and a pair of functions $(u, v)$
  solving \eqref{prob} in $\Ombar \times [0, \tmax)$ both weakly and classically in the sense of Definition~\ref{def:sol_concept}
  with the property that $\tmax$ is maximal, meaning that this solution cannot be extended beyond $\tmax$.
\end{prop}

In order to describe the critical mass phenomena in more detail, we henceforth let $\tmax(u_0, v_0)$ be as given by Proposition~\ref{prop:local_ex} and for $\Omega$ as in \eqref{eq:intro:omega} introduce the values
\begin{align}\label{eq:intro:m_lower_star}
  M_\star(\Omega) &\defs \sup \left\{ M > 0 \;\middle\vert\; \forall m \in (0, M]\ \forall u_0, v_0 \text{ fulfilling \eqref{eq:intro:init_reg} and } \intom u_0 = m:\, \tmax(u_0, v_0) = \infty \right\}
\intertext{as well as}\label{eq:intro:m_upper_star}
  M^\star(\Omega) &\defs \inf \left\{ M > 0 \;\middle\vert\; \forall m \ge M\ \exists u_0, v_0 \text{ fulfilling \eqref{eq:intro:init_reg} and } \intom u_0 = m:\, \tmax(u_0, v_0) < \infty \right\}.
\end{align}
That is, for all masses smaller than $M_\star(\Omega)$, all solutions exist globally, while for all masses larger than $M^\star(\Omega)$, initial data exist whose corresponding solutions blow up in finite time.
This leaves open the possibility that $M_\star(\Omega) < M^\star(\Omega)$. In this case, there would be an intermediate regime in which for some masses all solutions exist globally and for some masses they do not (cf.~\cite{FuhrmannEtAlDoubleCriticalMass2022}).

Moreover, when $\Omega$ is a disc, we may ask whether these values change when the situation is restricted to radially symmetric settings, i.e.\ require that
\begin{align}\label{eq:intro:init_rad_sym}
  u_0, v_0 \text{ are radially symmetric}.
\end{align}
To that end, we also set
\begin{align*}
  M_\star(\Omega, \textrm{rad}) &\defs \sup \left\{ M > 0 \;\middle\vert\; \forall m \in (0, M]\ \forall u_0, v_0 \text{ fulfilling \eqref{eq:intro:init_reg}, \eqref{eq:intro:init_rad_sym} and } \intom u_0 = m:\, \tmax(u_0, v_0) = \infty \right\}
\intertext{and}
  M^\star(\Omega, \textrm{rad}) &\defs \inf \left\{ M > 0 \;\middle\vert\; \forall m \ge M\ \exists u_0, v_0 \text{ fulfilling \eqref{eq:intro:init_reg}, \eqref{eq:intro:init_rad_sym} and } \intom u_0 = m:\, \tmax(u_0, v_0) < \infty \right\}.
\end{align*}
Several partial results for the size of these values are available.
If $\Omega$ is a smooth, bounded, planar domain, then $M_\star(\Omega) \ge 4\pi$ and $M_\star(\Omega, \textrm{rad}) \ge 8\pi$, cf.\ \cite{NagaiEtAlApplicationTrudingerMoserInequality1997}.
For such domains and all masses larger than $4\pi$ and not equaling an integer multiple of $4\pi$, unbounded solutions are constructed in \cite{HorstmannWangBlowupChemotaxisModel2001} (and also more recently in a different way in \cite{FujieJiangNoteConstructionNonnegative2022}) --
these, however, may potentially exist globally in time so that this result has no direct influence on $M^\star(\Omega)$.
If $\Omega = B_R(0) \subseteq \R^2$, $R > 0$, then $M^\star(\Omega, \textrm{rad}) \le 8\pi$ (and hence also $M^\star(\Omega) \le 8\pi$),
cf.\ \cite{MizoguchiWinklerBlowupTwodimensionalParabolic} (see also the earlier work \cite{HerreroVelazquezBlowupMechanismChemotaxis1997}).
Thus, if $\Omega$ is a disc, then $M_\star(\Omega, \textrm{rad}) = 8\pi = M^\star(\Omega, \textrm{rad})$,
while for arbitrary smooth, bounded, planar domains $\Omega$ it is up to now only known that $4 \pi \le M_\star(\Omega) \le M^\star(\Omega) \le 8 \pi$.
In \cite{NagaiEtAlChemotacticCollapseParabolic2000}, the critical mass identity $M_\star(\Omega) = 4\pi = M^\star(\Omega)$ is conjectured for smooth, bounded, planar domains $\Omega$.
Moreover, \cite[Theorem~1]{NagaiEtAlChemotacticCollapseParabolic2000} implies that blow-up has to occur at a single point at the boundary (or not at all) for masses between $4\pi$ and $8\pi$.
The actual occurrence of such boundary blow-up has, up to now, not been shown for \eqref{prob}.
 
As usual, more is known for parabolic--elliptic simplifications of \eqref{prob}, where the second equation is replaced by $0 = \Delta v - v + u$.
For such systems, namely, \cite[Theorem~2]{SenbaSuzukiChemotacticCollapseParabolicelliptic2001} and \cite[Theorem~3.2]{NagaiBlowupNonradialSolutions2001} show
$M_\star(\Omega) \ge 4\pi$ and $M^\star(\Omega) \le 4\pi$, respectively, for all smooth, bounded domains $\Omega \subseteq \R^2$, where the latter result additionally assumes that $\partial \Omega$ contains a line segment.
(For an early blow-up result in a related system, see also \cite{JagerLuckhausExplosionsSolutionsSystem1992}.)
Together, they imply $M_\star(\Omega) = 4\pi = M^\star(\Omega)$ for such domains.
Moreover, as in the fully parabolic case, $M_\star(\Omega, \textrm{rad}) = 8\pi = M^\star(\Omega, \textrm{rad})$ holds whenever $\Omega$ is a disc (cf.\ \cite[Theorem~2.1 and Corollary~3.1]{NagaiBlowupRadiallySymmetric1995}).
For a detailed discussion of the parabolic--elliptic system, we refer to \cite{SuzukiFreeEnergySelfinteracting2005}.
Let us also briefly mention that critical mass phenomena (with slightly different flavors) have also been detected
if $\Omega = \R^2$ (\cite{BlanchetEtAlTwodimensionalKellerSegelModel2006}, \cite{WeiGlobalWellposednessBlowup2018}), if fluid-interaction is accounted for (\cite{HeTadmor,win_fluidbu}),
if the signal is produced indirectly (\cite{TaoWinklerCriticalMassInfinitetime2017})
or if different boundary conditions are imposed on $v$ (\cite{BilerNadziejaExistenceNonexistenceSolutions1994}, \cite{FuhrmannEtAlDoubleCriticalMass2022}).

These results naturally lead to the question whether there is a domain $\Omega$ with $M_\star(\Omega) = M^\star(\Omega)$;
that is, whether there is a critical mass distinguishing between global existence and finite-time blow-up also for the fully parabolic system \eqref{prob} in non-radially symmetric settings
(and if there is, what its precise value is).
Our main results states that for certain domains $\Omega$, we can indeed guarantee that $M_\star(\Omega) = M^\star(\Omega)$.
For instance, $2\pi$ is the critical mass for quarter discs.
\begin{theorem}\label{th:critical_mass}
  Let $\Omega\subset \R^2$ be a circular sector with central angle $\theta \in (0, \frac{\pi}{2}]$ (and arbitrary positive finite radius).
  Then
  \begin{align*}
    M_\star(\Omega) = 4\theta = M^\star(\Omega).
  \end{align*}
\end{theorem}

The proof of Theorem~\ref{th:critical_mass} is split into two parts.
We first show in Section~\ref{sec:ge} that the solution $(u, v)$ constructed in Proposition~\ref{prop:local_ex} based on \cite{GajewskiZachariasGlobalBehaviourReactiondiffusion1998} is global whenever the initial mass is sufficiently small.
This is achieved by adapting the proof in \cite{NagaiEtAlApplicationTrudingerMoserInequality1997} to non-smooth domains
and thereby by crucially relying on the Trudinger--Moser inequality (cf.\ \cite[Proposition~2.3]{ChangYangConformalDeformationMetrics1988})
whose constants depend on the minimal interior angle of the domain.
\begin{theorem}\label{th:ge}
  Suppose \eqref{eq:intro:omega} and denote the minimal interior angle of $\Omega$ by $\theta$ (and set $\theta = \pi$ if there are no corners).
  Then
  \begin{align*}
    M_\star(\Omega) \ge 4\theta.
  \end{align*}
\end{theorem}

Next, in Section~\ref{sec:blow_up}, if $\Omega$ is a circular sector and given sufficiently large initial mass, we construct initial data such that the corresponding solutions blow up in finite time.
This is achieved by restricting finite-time blow-up solutions on discs (constructed in \cite{MizoguchiWinklerBlowupTwodimensionalParabolic}) to $\Omega$.
\begin{theorem}\label{th:blow_up}
  Let $\Omega\subset \R^2$ be a circular sector with central angle $\theta \in (0, 2\pi)$ (and arbitrary positive finite radius).
  Then
  \begin{align*}
    M^\star(\Omega) \le 4\theta.
  \end{align*}
  Moreover, for each $m > 4\theta$, one can choose $u_0, v_0$ fulfilling \eqref{eq:intro:init_reg}, $\intom u_0 = m$ and $\tmax(u_0, v_0) < \infty$
  such that blow-up occurs at $0 \in \partial \Omega$ in the sense that there are $(x_k)_{k \in \N} \subseteq \Ombar$ and $(t_k)_{k \in \N} \subseteq [0, \tmax(u_0, v_0))$
  such that $x_k \to 0$, $t_k \nea \tmax$ and $u(x_k, t_k) \to \infty$ as $k \to \infty$.
\end{theorem}

Theorem~\ref{th:blow_up} appears to be the first result for \eqref{prob} (or any fully parabolic chemotaxis system considered on a bounded domain),
where blow-up is shown to take place on the boundary.
For most choices of $\theta \in (0, 2\pi)$, Theorem~\ref{th:blow_up} shows that corners of $\Omega$ may be blow-up points,
but the choice $\theta=\pi$ makes it clear that blow-up can also happen at smooth parts of the boundary.

We also note that our technique fails if one considers \eqref{prob} with (no-flux boundary conditions for $u$ and) homogeneous Dirichlet boundary conditions for $v$.
However, for such a system blow-up at a boundary point probably should not be expected at all: 
At least for certain parabolic--elliptic simplifications all blow-up points have been proven to be contained in the interior of the domain (\cite{SuzukiExclusionBoundaryBlowup2013}).

\section{Local existence: Proof of Proposition~\ref{prop:local_ex}}\label{sec:local_ex}
Throughout this section, we fix a domain $\Omega$ fulfilling \eqref{eq:intro:omega} and $u_0, v_0$ satisfying \eqref{eq:intro:init_reg}.

\begin{definition}\label{def:sol_concept}
  Let $T \in (0, \infty]$.
  \begin{itemize}
    \item 
      A pair of nonnegative functions $(u, v)$ is called a \emph{weak solution} of \eqref{prob} (in $\Ombar \times [0, T)$) if
      \begin{align}\label{eq:local_ex:reg_u}
        u &\in L_{\loc}^\infty(\Ombar \times [0, T)) \cap L_{\loc}^2([0, T); \sob12), \\
        v &\in L_{\loc}^\infty(\Ombar \times [0, T)) \cap C^0([0, T); \sob12), \label{eq:local_ex:reg_v}
      \end{align}
      the function $(0, T) \ni t \mapsto \|\nabla v(t)\|_{\leb2}^2$ is absolutely continuous,
      \begin{align}\label{eq:local_ex:reg_ut_vt}
        (u_t, v_t) \in L_{\loc}^2([0, T); \dual{\sob12}) \times L_{\loc}^2([0, T); \leb2)
      \end{align}
      and $u, v$ fulfill $u(\cdot, 0) = u_0$ and $v(\cdot, 0) = v_0$ a.e.\ in $\Omega$ as well as
      \begin{align}
        \int_0^{t} \intom u_t \varphi &= - \int_0^{t} \intom (\nabla u - u \nabla v) \cdot \nabla \varphi \quad \text{and} \notag \\
        \int_0^{t} \intom v_t \varphi &= - \int_0^{t} \intom \nabla v \cdot \nabla \varphi - \int_0^{t} \intom v \varphi + \int_0^{t} \intom u \varphi \label{eq:sol_concept:v_eq}
      \end{align}
      for all $\varphi \in L^2((0, T); \sob12)$ and all $t \in (0, T)$.

    \item
      A pair of nonnegative functions $(u, v)$ is called a \emph{classical solution}  of \eqref{prob} if
      \begin{align*}
        u,v \in C^0((\Ombar \setminus V) \times [0, T)) \cap C^{2, 1}((\Ombar \setminus V) \times (0, T)),
      \end{align*}
      where $V$ denotes the set of vertices of $\Omega$,
      and the equations in \eqref{prob} are fulfilled pointwise in $(\Ombar \setminus V) \times [0, T)$.
  \end{itemize}
\end{definition}

\begin{remark}
  Let $T\in(0,\infty)$ and let $(u, v)$ be a weak solution of \eqref{prob}. Then \eqref{eq:local_ex:reg_u}--\eqref{eq:sol_concept:v_eq} assert $\Delta v \in L^2(\Omega \times (0, T))$,
  hence \eqref{eq:local_ex:reg_v} and \eqref{eq:sol_concept:v_eq} imply $\frac12 \ddt \intom |\nabla v|^2 = -\intom v_t \Delta v$ a.e.\ in $(0, T)$.
\end{remark}

The following local existence result and the extensibility criterion is essentially due to \cite{GajewskiZachariasGlobalBehaviourReactiondiffusion1998}.
We note that a local existence result for less regular initial data has been proven in \cite[Theorem~1]{BilerLocalGlobalSolvability1998}.
\begin{lemma}\label{lm:local_ex}
  There exist $\tmax \in (0, \infty]$ and a unique weak solution of \eqref{prob} in $\Ombar \times [0, \tmax)$ in the sense of Definition~\ref{def:sol_concept}.
  Moreover, if $\tmax$ is chosen to be as large as possible, then 
  \begin{align}\label{eq:local_ex:ext_crit}
    \tmax < \infty
    \quad \text{implies} \quad
    \limsup_{t \nea \tmax} \left( \intom (u \ln u)(\cdot, t) + \intnstom v_t^2 \right) = \infty.
  \end{align}
\end{lemma}
\begin{proof}
Let $p>2$ such that $v_0 \in \sob1p$. By the construction in \cite[Theorem~3.3]{GajewskiZachariasGlobalBehaviourReactiondiffusion1998}, for sufficiently small $T>0$ functions $(u,v)\in(C([0,T];L^2(\Omega)))$ can be found which constitute a weak solution of \eqref{prob} in $\Ombar\times[0,T)$ and which satisfy $u(\cdot,T)\in L^\infty(\Omega)$ and $v(\cdot,T)\in W^{1,p}(\Omega)$. With $(u,v)(\cdot,T)$ as initial data, the process can be repeated. We suppose that $\tmax$ is chosen maximally. From the construction in \cite[part~(iii) of the proof of Theorem~3.3]{GajewskiZachariasGlobalBehaviourReactiondiffusion1998}, this means that either $T=\infty$ or one of  
\begin{equation}\label{eq:unboundedGZ}
\int_0^t \norm[\Lom p]{\ure^{-v(\cdot,s)}u(\cdot,s)}\ds,\quad \norm[\Lom\infty]{v(\cdot,t)}, \quad \norm[\Lom2]{\nabla v(\cdot,t)} \quad \text{or} \quad \norm[\Lom2]{\ure^{-v(\cdot,t)}u(\cdot,t)}
\end{equation}
has to be unbounded as $t \nea \tmax$. In order to show \eqref{eq:local_ex:ext_crit}, let us assume that $\tmax<\infty$ and that there exists $c_1 > 0$ such that
  \begin{align*}
    \intom (u \ln u)(\cdot, t) + \intnstom v_t^2 \le c_2
    \qquad \text{for all $t \in (0, \tmax)$}.
  \end{align*}
  By \cite[Lemma~4.10 and Lemma~4.11]{GajewskiZachariasGlobalBehaviourReactiondiffusion1998}, both $u$ and $v$ are then bounded in $\Omega \times (0, \tmax)$ and hence so are all quantities in \eqref{eq:unboundedGZ}.
  Finally, testing the first and second equation with $u_-$ and $v_-$, respectively, yields nonnegativity of both $u$ and $v$ a.e.
\end{proof}

Next, we make use of parabolic regularity theory to show that the solution constructed in Lemma~\ref{lm:local_ex} is also a classical solution.
\begin{lemma}\label{lm:smooth_outside_corners}
  Let $(u, v)$ be the weak solution of \eqref{prob} given by Lemma~\ref{lm:local_ex} with maximal existence time $\tmax$.
  Then $(u, v)$ is also a classical solution in the sense of Definition~\ref{def:sol_concept}.
\end{lemma}
\begin{proof}
  For every $\psi \in \con2$ with $\partial_\nu \psi = 0$ on $\partial \Omega$ and every $\varphi \in \con1$, direct computations show that
  \begin{align*}
        \intntom (\psi u)_t \varphi
    &=  - \intntom (\nabla(\psi u) - 2 u \nabla \psi - \psi u \nabla v) \cdot \nabla \varphi
        + \intntom (u \Delta \psi + u \nabla \psi \cdot \nabla v) \varphi \quad \text{and} \\
        \intntom (\psi v)_t \varphi
    &=  - \intntom (\nabla(\psi v) - 2 v \nabla \psi) \cdot \nabla \varphi
        + \intntom (v \Delta \psi - v \psi + u \psi) \varphi
  \end{align*}
  hold for all $T\in(0,\tmax)$. 
  The main idea is now to fix a point $x \in \Ombar \setminus V$, where $V$ denotes the set of vertices of $\Omega$, and $0 < \tau < T < \tmax$,
  choose a cut-off function $\psi \in \con2$ with $\partial_\nu \psi= 0$ on $\partial\Ombar$ which equals $1$ in a neighbourhood of $x$ and vanishes on a neighbourhood of $V$
  (existence of such $\psi$ can be shown by, e.g., following the proof of \cite[Lemma~3.2]{BlackEtAlPossiblePointsBlowup2022})
  and then to apply various parabolic regularity results to $\psi u$ and to $\psi v$ which solve certain parabolic equations on smooth domains containing $\supp \psi$.
  In each step, we may choose a new cut-off function $\tilde \psi$ with smaller support so that regularity information on, say, $\nabla (\psi v)$ translates to information on $\nabla v$ on all of $\supp \tilde \psi$.
  
  Indeed, in this way \cite[Theorem~1.3]{PorzioVespriHolderEstimatesLocal1993} first shows that $v$ is Hölder continuous near $S_x \defs \{x\} \times (\tau, T)$,
  whenceupon \cite[Theorem~1.1]{LiebermanHolderContinuityGradient1987} asserts Hölder continuity also of $\nabla v$ near $S_x$.
  Then we can again make use of \cite[Theorem~1.3]{PorzioVespriHolderEstimatesLocal1993} and \cite[Theorem~1.1]{LiebermanHolderContinuityGradient1987}
  to first obtain that $u$ and then also that $\nabla u$ is Hölder continuous near $S_x$.
  Thus, two applications of \cite[Theorem IV.5.3]{LadyzenskajaEtAlLinearQuasilinearEquations1988} yield $u, v \in C^{2, 1}((\Ombar \setminus V) \times (0, \tmax))$.
  As both $u$ and $v$ are also continuous in $[0, \tmax)$ as functions with values in $L^\infty$ by Lemma~\ref{lm:local_ex},
  they also belong to $C^0((\Ombar \setminus V) \times [0, \tmax))$.
  Finally, three testing procedures with test functions supported in $(0, \tmax) \times \Omega$, near $(0, \tmax) \times (\Ombar \setminus V)$ and near $\{0\} \times \Omega$, respectively, show that $(u, v)$ also solves \eqref{prob} classically in $(\Ombar \setminus V) \times [0, \tmax)$.
\end{proof}

\begin{proof}[Proof of Proposition~\ref{prop:local_ex}]
  In Lemma~\ref{lm:local_ex}, a local weak solution of \eqref{prob} has been constructed, which is also a classical solution by Lemma~\ref{lm:smooth_outside_corners}.
\end{proof}

\section{Global existence: Proof of Theorem~\ref{th:ge}}\label{sec:ge}
Throughout this section, we assume \eqref{eq:intro:omega}
and denote the minimal interior angle of $\Omega$ by $\theta$ (and set $\theta = \pi$ if there are no corners).
Moreover, we fix $u_0, v_0$ fulfilling \eqref{eq:intro:init_reg} as well as the solution $(u, v)$ of \eqref{prob} and its maximal existence time $\tmax$ given by Proposition~\ref{prop:local_ex}.

We shall show $\tmax=\infty$ for sufficiently small $\intom u_0$,
which we will achieve by following the reasoning in \cite{NagaiEtAlApplicationTrudingerMoserInequality1997} and \cite{GajewskiZachariasGlobalBehaviourReactiondiffusion1998}. (We also refer to \cite{BilerLocalGlobalSolvability1998} where these techniques have been adapted to systems with different boundary conditions for $v$.)

The approach rests on two crucial observations. 
The first one, which has been first noted in \cite{NagaiEtAlApplicationTrudingerMoserInequality1997} and \cite{GajewskiZachariasGlobalBehaviourReactiondiffusion1998},
is that \eqref{prob} admits an energy-type functional.
\begin{lemma}\label{lm:energy}
  For all $t \in (0, \tmax)$, the estimate
  \begin{align*}
    &\pe  \intom u(t) \ln u(t) - \intom u(t)v(t) + \frac12 \intom v^2(t) + \frac12 \intom |\nabla v(t)|^2
    +   \int_0^t \intom u |\nabla (\ln u - v)|^2 + \int_0^t \intom v_t^2 \\
    &\le \intom u_0 \ln u_0 - \intom u_0v_0 + \frac12 \intom v^2_0 + \frac12 \intom |\nabla v_0|^2
  \end{align*}
  holds.
\end{lemma}
\begin{proof}
  This follows by a direct calculation, see for instance \cite[Lemma~3.3]{NagaiEtAlApplicationTrudingerMoserInequality1997}.
\end{proof}

Second, we will rely on the following consequence of the Trudinger--Moser inequality.
\begin{lemma}\label{lm:trudinger_moser}
  Then there exists $C_\Omega > 0$ such that
  \begin{align*}
    \intom \ure^{|\varphi|} \le C_\Omega \exp \left( \frac1{8\theta} \intom |\nabla \varphi|^2 + \frac1{|\Omega|} \intom |\varphi| \right)
    \qquad \text{for all $\varphi \in \sob12$}.
  \end{align*}
\end{lemma}
\begin{proof}
  As shown for instance in \cite[Corollary~2.7]{GajewskiZachariasGlobalBehaviourReactiondiffusion1998},
  this follows from the Trudinger--Moser inequality proved in \cite[Theorem~1.2]{CianchiMoserTrudingerInequalitiesBoundary2005} (see also \cite[Proposition~2.3]{ChangYangConformalDeformationMetrics1988}).
\end{proof}

Similarly as in \cite[Lemma~3.4]{NagaiEtAlApplicationTrudingerMoserInequality1997},
these preparations allow us to show that solutions exist globally whenever the mass of the first component is sufficiently small.
\begin{lemma}\label{lm:tmax_infty}
  Suppose $m \defs \intom u_0 < 4\theta$.
  Then $\tmax \defs \tmax(u_0, v_0) = \infty$.
\end{lemma}
\begin{proof}
  Since $m < 4\theta$, we can find $\eta > 0$ such that $\frac{(1+\eta)m}{8\theta} < \frac12$.
  As integrating the first equation in \eqref{prob} shows $\intom u = m$ in $(0, \tmax)$,
  applying Jensen's inequality and Lemma~\ref{lm:trudinger_moser} yields
  \begin{align*}
          (1+\eta) \intom uv -\intom u \ln u
    &=    m \intom \frac{u}{m} \ln \left(\frac{\ure^{(1+\eta)v}}{u}\right)
     \le  m \ln\left(\intom \frac{u}{m} \frac{\ure^{(1+\eta)v}}{u}\right)
     =    m \ln\left(\frac1m \intom \ure^{(1+\eta)v}\right) \\
    &\le  m \ln\left(\frac{C_\Omega}{m} \exp \left( \frac{1+\eta}{8\theta} \intom |\nabla v|^2 + \frac{1+\eta} {|\Omega|} \intom v \right) \right) \\
    &\le  m \ln\left(\frac{C_\Omega}{m}\right) + \frac12 \intom |\nabla v|^2 + \frac{(1+\eta)m}{|\Omega|} \intom v
    \qquad \text{in $(0, \tmax)$},
  \end{align*}
  where in the last step we have used $\frac{(1+\eta)m}{8\theta} < \frac12$.
  As testing the second equation in \eqref{prob} with $1$ reveals $\intom v \le \max\{\intom v_0, m\}$, we thus conclude that there is $c_1 > 0$ such that
  \begin{align*}
        \intom uv
    \le \frac{1}{1+\eta} \intom u \ln u + \frac{1}{2(1+\eta)} \intom |\nabla v|^2 + c_1
    \qquad \text{in $(0, \tmax)$}.
  \end{align*}
  Together with Lemma~\ref{lm:energy}, this shows boundedness of $\intom u \ln u$ in $(0, \tmax)$.
  As $\int_0^{\tmax} \intom v_t^2$ is also bounded by Lemma~\ref{lm:energy},
  we can conclude from \eqref{eq:local_ex:ext_crit} that $\tmax = \infty$.
\end{proof}

\begin{proof}[Proof of Theorem~\ref{th:ge}]
  According to the definition of $M_\star(\Omega)$ in \eqref{eq:intro:m_lower_star}, $M_\star(\Omega) \ge 4\theta$ immediately follows from Lemma~\ref{lm:tmax_infty}.
\end{proof}

\section{Finite-time blow-up: Proof of Theorem~\ref{th:blow_up} and Theorem~\ref{th:critical_mass}}\label{sec:blow_up}
In this final section, we prove Theorem~\ref{th:blow_up}. 
As crucial ingredient we use the known existence of blow-up solutions on a disk and to this aim recall \cite[Proposition~1.3]{MizoguchiWinklerBlowupTwodimensionalParabolic}:
\begin{lemma}\label{lem:mizowin}
 Let $\Omega=B_R(0)\subset ℝ^2$ with some $R>0$, and let $m>8π$. Then for all $T>0$ and each $p>1$ there exist $ε>0$ and $(u_0,v_0)\in \mc I\coloneqq\set{(\hat u_0,\hat v_0)\in C^0(\Ombar)\times W^{1,\infty}(\Omega)\mid \hat u_0 \text{ and } \hat v_0 \text{ are radially symmetric and positive}$ $\text{in } \Ombar}$ such that $\intom u_0=m$ and such that the solutions of \eqref{prob} for all initial data $(\wt{u}_0,\wt{v}_0)\in \mc I$ satisfying $\norm[\Lom p]{\wt{u}_0-u_0}+\norm[W^{1,2}(\Omega)]{\tilde{v_0}-v_0}<ε$ lead to solutions blowing up before time $T$.
\end{lemma}

In restricting these solutions to circular sectors, we have to ensure that the boundary conditions are still satisfied, and note the following elementary fact:
\begin{lemma}\label{lm:neumann_radially_sym}
  Let $R > 0$, $\xi, \nu \in \partial B_1(0)$ with $\xi \cdot \nu = 0$
  and $\varphi \in C^1(B_R(0))$ be radially symmetric.
  Then
  \begin{align*}
    \nabla \varphi(x) \cdot \nu = 0 
    \qquad \text{for all $x \in L \defs \R \xi \cap B_R(0)$}.
  \end{align*}
\end{lemma}
\begin{proof}
  Without loss of generality, we may assume $\xi = (1, 0)$ and $\nu = (0, 1)$.
  Then $L = \{\,(x_1, x_2) \in B_R(0) : x_2 = 0\,\}$ 
  and the radial symmetry of $\varphi$ implies
  $\varphi(x_1, x_2) = \varphi(x_1, -x_2)$
  for all $(x_1, x_2) \in B_R(0)$;
  that is, $(0, \sqrt{R^2 - x_1^2}) \ni h \mapsto \varphi(x_1, h)$ is an even function for all $x_1 \in (0, R)$ and hence has derivative $0$ at $0$.
\end{proof}

\begin{lemma}\label{lm:blowup}
  Let $\Omega$ be a circular sector with central angle $\theta \in (0, 2\pi)$ and radius $R > 0$
  and let $m > 4 \theta$.
  Then there exist $u_0, v_0$ satisfying \eqref{eq:intro:init_reg} and $\intom u_0 = m$ such that the solution of \eqref{prob} given by Proposition~\ref{prop:local_ex} blows up in finite time.
\end{lemma}
\begin{proof}
  As $\wt m \defs \frac{2\pi}{\theta} m > 8 \pi$, based on Lemma~\ref{lem:mizowin} we let $(\wt u,\wt v)$, with $ (\wt u_0,\wt v_0)\coloneqq (\wt u(\cdot,0),\wt v(\cdot,0))$ satisfying $\intom u_0=\wt m$, be radially symmetric classical solutions of \eqref{prob} on $\wt\Omega \coloneqq B_R(0)$ blowing up at some finite time $T\in(0,\infty)$.  
  Moreover, \cite[Theorem~3]{NagaiEtAlChemotacticCollapseParabolic2000} asserts that $0$ is the only blow-up point of $\wt u$. 
  We set $(u_0, v_0) \defs (\wt u_0, \wt v_0){\big|}_{\Ombar}$ and note that due to the radial symmetry of $u_0$ we have
  \begin{align*}
    \intom u_0 = \frac{\theta}{2\pi} \int_{\tilde \Omega} \wt u_0 = \frac{\theta}{2\pi} \wt m = m
  \end{align*}
  and that $0$ is also a blow-up point of $u$.
  We next claim that $(u, v) \defs (\wt u, \wt v){\big|}_{\Ombar \times [0, T)}$ is a classical solution of \eqref{prob}:
  That the differential equations, the initial conditions and the boundary conditions for the circular arc are fulfilled follows immediately
  from the fact that $(u, v)$ is a solution of \eqref{prob} in $\wt \Omega \times [0, T)$.
  Finally, Lemma~\ref{lm:neumann_radially_sym} shows that the boundary conditions are also fulfilled on the remaining (smooth part of the) boundary.
  Therefore, $(u, v)$ is a solution of \eqref{prob} in the sense of Definition~\ref{def:sol_concept}
  and due to the uniqueness statement in Lemma~\ref{lm:local_ex} thus has to coincide with the solution given by Proposition~\ref{prop:local_ex}.
\end{proof}

\begin{proof}[Proof of Theorem~\ref{th:blow_up}]
  This is a direct consequence of \eqref{eq:intro:m_upper_star} and Lemma~\ref{lm:blowup}.
\end{proof}

At last, we note that Theorem~\ref{th:ge} and Theorem~\ref{th:blow_up} entail Theorem~\ref{th:critical_mass}.
\begin{proof}[Proof of Theorem~\ref{th:critical_mass}]
  As the minimal interior angle of a circular sector with central angle $\theta$ is $\min\{\theta, \frac{\pi}{2}\}$,
  Theorem~\ref{th:critical_mass} results as a combination of Theorem~\ref{th:ge} and Theorem~\ref{th:blow_up}.
\end{proof}


\end{document}